\documentclass[11pt]{amsart}

\usepackage{amsmath,amssymb,amsthm,mathtools,mathrsfs}
\mathtoolsset{showonlyrefs}
\usepackage[margin=1.15in]{geometry}
\usepackage{tikz}
\usepackage[colorlinks=true,citecolor=blue,linkcolor=blue,urlcolor=blue]{hyperref}

\numberwithin{equation}{section}

\newtheorem{theorem}{Theorem}[section]
\newtheorem{proposition}[theorem]{Proposition}
\newtheorem{lemma}[theorem]{Lemma}

\theoremstyle{definition}
\newtheorem{definition}[theorem]{Definition}
\newtheorem{remark}[theorem]{Remark}

\begin{document}
\title[A Horizon-Free Extrinsic Penrose Inequality]{A Horizon-Free Extrinsic Penrose Inequality}
\date{\today}

\author{Caiyan Li}
\address[C.L]{School of Mathematical Sciences, Xiamen University, 361005, Xiamen, P.R. China}
\email{caiyanli@xmu.edu.cn}

\begin{abstract}
Let $S\subset\mathbb R^3$ be a properly embedded mean-convex planar surface with finitely many ends. Designate one end as asymptotically flat, assume that $H_S$ is integrable there, and denote its extrinsic mass by $m_+(S)$. Let $A_S$ be the infimum of the areas of compact surfaces separating the distinguished end from all the others. We prove
\[
m_+(S)\geq\sqrt{\frac{A_S}{\pi}}.
\]
No outermost free-boundary minimal surface is assumed, and no asymptotic or integrability condition is imposed on the other ends. If $A_S=0$, equality holds precisely for the Euclidean half-space. The complete catenoid realizes equality with $A_S>0$. Conversely, if equality holds with $A_S>0$, then $A_S$ is attained by a flat free-boundary disk that is outermost toward the distinguished end, and the corresponding component of $S$ is a half-catenoid.
\end{abstract}

\maketitle

\section{Introduction}
\label{sec:introduction}

The positive mass theorem and the Riemannian Penrose inequality are the basic mass inequalities for asymptotically flat manifolds. Schoen and Yau proved that nonnegative scalar curvature implies nonnegative ADM mass, with equality only for Euclidean space \cite{SchoenYau79,SchoenYau81}. In dimension three, the Penrose inequality strengthens this conclusion in the presence of an outermost minimal surface $D$ to $m_{\mathrm{ADM}}\geq\sqrt{|D|/(16\pi)}$. Huisken--Ilmanen proved the connected case by developing weak inverse mean curvature flow, while Bray treated arbitrary numbers of components using conformal flow \cite{HI,Bray}.

The theory was subsequently extended to manifolds with noncompact boundary. Almaraz, Barbosa, and de Lima proved the corresponding positive mass theorem \cite{AlmarazBarbosaLima}. Marquardt constructed weak inverse mean curvature flow with free boundary and identified a monotone boundary analogue of the Hawking mass \cite{Marquardt}. Koerber used this theory to prove the Penrose inequality in dimension three \cite{Koerber}, and Eichmair and Koerber later obtained a doubling proof and a higher-dimensional extension \cite{EKdoubling}. These results retain an outermost minimal surface as the horizon. Zhu introduced a different formulation in the boundaryless setting: the horizon area is replaced by the infimum of the areas of compact surfaces separating a distinguished end from the other ends \cite{Zhu}. His inequality therefore requires no prescribed horizon.

The extrinsic theory began with Volkmann's asymptotically flat support surfaces and their exterior mass \cite{Volkmann}. His thesis proves positive-mass rigidity under asymptotically catenoidal assumptions and records Huisken's conjecture for an extrinsic Penrose inequality. Eichmair and Koerber resolved this conjecture for an asymptotically flat support surface carrying an outermost free-boundary minimal surface $D$, proving $m\geq\sqrt{|D|/\pi}$ with equality precisely when the exterior support is a half-catenoid \cite{EK}. Their proof uses $J_t$-minimizing minimal capillary surfaces for
\begin{equation}
 J_t(\Sigma)=|\Sigma|-\tanh(t)|S(\Sigma)|.
 \label{eq:EK-free-energy}
\end{equation}
Here $S(\Sigma)$ is the compact lateral support region determined by $\Sigma$. They establish monotonicity of the associated free-energy mass. For related regularity and stability results for capillary surfaces, see \cite{ChodoshEdelenLi,DePhilippisMaggi,HongSaturnino,Ros}. The present work combines Zhu's separating-area formulation with this capillary method. It removes the prescribed horizon and allows finitely many additional ends without asymptotic control.

Designate one end $\mathcal E_+$ as asymptotically flat, and let $A_S$ be the infimum of the areas of surfaces separating $\mathcal E_+$ from the remaining ends. Under
\[
 H_S\geq0\quad\text{on }S,
 \qquad H_S\in L^1(S_+),
\]
we prove
\[
 m_+(S)\geq\sqrt{\frac{A_S}{\pi}},
\]
with no asymptotic or integrability assumption on the other ends. If $A_S=0$, equality characterizes the Euclidean half-space. If $A_S>0$, equality determines the distinguished exterior as a half-catenoid.

Let $M\subset\mathbb R^3$ be a connected domain with connected smooth
properly embedded boundary $S=\partial M$. Let $\nu_S$ be its outward
unit normal. Our sign convention is
\[
 h_S(X,Y)=\langle\bar\nabla_X\nu_S,Y\rangle,
 \qquad H_S=\operatorname{tr}_S h_S.
\]
For some integer $q\geq0$, choose distinct points
$p_+,p_1,\ldots,p_q\in\partial\mathbb B^3$, with only $p_+$ when $q=0$,
and assume that there is a homeomorphism of pairs
\begin{equation}
 (\overline M,S)\cong
 \left(\overline{\mathbb B^3}\setminus\{p_+,p_1,\ldots,p_q\},
 \partial\mathbb B^3\setminus\{p_+,p_1,\ldots,p_q\}\right).
 \label{eq:topology}
\end{equation}
Thus $S$ is planar and has $q+1$ ends. The end corresponding to $p_+$ is
denoted by $\mathcal E_+$, and the remaining ends by
$\mathcal E_1,\ldots,\mathcal E_q$. Condition~\eqref{eq:topology} is purely
topological; no geometric control is imposed at the latter ends.

\begin{definition}
\label{def:af-end}
The end $\mathcal E_+$ is \emph{asymptotically flat} if, after a rigid motion, there
are $R_0>0$, $\tau>1/2$, a compact set $K\subset S$, and a smooth function
$\psi:\mathbb R^2\setminus B_{R_0}\to\mathbb R$ such that
\[
 S_+=\{(y,\psi(y)):|y|>R_0\},
\]
where $S_+$ is the component of $S\setminus K$ representing
$\mathcal E_+$. Along this component, the chosen outward normal is
\[
 \nu_S(y,\psi(y))
 =\frac{(D\psi(y),-1)}{\sqrt{1+|D\psi(y)|^2}}.
\]
Moreover,
\[
 |D\psi(y)|+|y||D^2\psi(y)|
 =O(|y|^{-\tau}).
\]
With this orientation, the \emph{extrinsic mass of the distinguished end} is
\begin{equation}
 m_+(S)=\lim_{r\to\infty}\frac1{2\pi r}
 \int_{|y|=r}y\cdot D\psi\,d\mathcal H^1.
 \label{eq:mass}
\end{equation}
\end{definition}

Along $S_+$, the chosen orientation means that $M$ lies locally on the side
$x_3>\psi(y)$. The assumption $H_S\in L^1(S_+)$ ensures that the limit in
\eqref{eq:mass} exists; see the discussion preceding
\cite[Theorem~1]{EK}.

\begin{definition}
\label{def:admissible}
A compact smooth two-sided properly embedded surface
$\Sigma\subset\overline M$ is an \emph{admissible separator} if
$\partial\Sigma=\Sigma\cap S$, the intersection is transverse, no component
of $\Sigma$ is closed in $M$, and there is a relatively open set
$\Omega(\Sigma)\subset\overline M$ whose boundary in $\overline M$ is
$\Sigma$ and which satisfies
\[
 \mathcal E_j\subset\Omega(\Sigma)\quad(1\leq j\leq q),
 \qquad \mathcal E_+\cap\Omega(\Sigma)=\varnothing.
\]
Both $\Sigma$ and $\Omega(\Sigma)$ may be disconnected. The
collection of admissible separators is denoted by $\mathcal A_S$, and
\begin{equation}
 A_S=\inf_{\Sigma\in\mathcal A_S}|\Sigma|.
 \label{eq:AS}
\end{equation}
\end{definition}
The class $\mathcal A_S$ is nonempty by \eqref{eq:topology}.

For admissible disks $\Gamma$ and $\Gamma'$, we write $\Gamma\prec\Gamma'$
if $\Omega(\Gamma)\subset\Omega(\Gamma')$.

For $m>0$, write
\[
 \mathrm{Cat}_m=\left\{(y,x_3):|y|=m\cosh\left(\frac{x_3}{m}\right)\right\}
\]
for the complete catenoid of waist radius $m$, and
\begin{equation}
 \mathrm{Cat}_m^+=\mathrm{Cat}_m\cap\{x_3\geq0\}
 =\left\{\left(y,m\operatorname{arcosh}\frac{|y|}{m}\right):
 |y|\geq m\right\}
 \label{eq:half-catenoid}
\end{equation}
for the upper half-catenoid.

\begin{definition}
\label{def:outermost}
A free-boundary minimal surface $D\in\mathcal A_S$ is \emph{outermost toward
$\mathcal E_+$} if every component of $D$ belongs to the boundary of the
$\mathcal E_+$-component $M_+(D)$ of $M\setminus D$, and every compact
free-boundary minimal surface contained in $\overline{M_+(D)}$ is a
component of $D$.
\end{definition}

\begin{theorem}
\label{thm:main}
Let $S\subset\mathbb R^3$ satisfy \eqref{eq:topology}, and suppose that
$\mathcal E_+$ is asymptotically flat in the sense of
Definition~\ref{def:af-end}. Assume
 \begin{equation}
 H_S\geq0\quad\text{on }S,
 \qquad H_S\in L^1(S_+).
 \label{eq:natural-hypotheses}
\end{equation}
Then the following statements hold.
If $A_S=0$, then $m_+(S)\geq0$, with equality if and only if $q=0$ and,
after a rigid motion,
\[
 M=\{x_3>0\},\qquad S=\{x_3=0\}.
\]
If $A_S>0$, then
\begin{equation}
 m_+(S)\geq\sqrt{\frac{A_S}{\pi}}.
 \label{eq:main-inequality}
\end{equation}
If equality holds in \eqref{eq:main-inequality}, then $A_S$ is attained by a
free-boundary minimal disk $D$ that is outermost toward $\mathcal E_+$. After
a rigid motion,
\[
 D=\{(y,0):|y|\leq m_+(S)\},
\]
and the component of $S\setminus\partial D$ containing $\mathcal E_+$ is
$\mathrm{Cat}_{m_+(S)}^+$.
\end{theorem}

For an open set $U\subset\mathbb R^3$ and a measurable set $E\subset U$, let $P(E;U)$ denote the perimeter of $E$ in $U$. For a finite-perimeter set $E\subset M$, let
$\operatorname{Tr}_S\chi_E$ denote its interior trace on $S$. Extending $E$ by zero across $S$, the perimeter decomposition
\cite[Chapter~16]{Maggi} gives
\begin{equation}
 P(E;\mathbb R^3)
 =P(E;\mathring M)+\int_S\operatorname{Tr}_S\chi_E\,dA_S.
 \label{eq:trace-perimeter-completion}
\end{equation}

\begin{definition}
\label{def:inner-end-region}
A measurable set $E\subset M$ is an \emph{inner end region} if it has locally finite perimeter, contains a full tail
of every $\mathcal E_j$, $1\leq j\leq q$, and excludes a full tail of $\mathcal E_+$.
\end{definition}

When $A_S=0$, the conclusion and its equality case follow from the positive
mass theorem for arbitrary ends and noncompact boundary
\cite{PMArbitraryEnds}. We therefore consider $A_S>0$ below; in this case,
$q\geq1$.

We now outline the proof for $A_S>0$.
The main compactness issue is escape through the uncontrolled ends. After
constructing calibrated admissible disks $\Gamma_0\prec\Gamma_1$, associate
with each inner end region $E$ the set
\[
G=\Omega(\Gamma_1)\setminus E.
\]
Under this correspondence, the capillary energy changes only by an additive
constant, while the functional on $G$ satisfies
\begin{equation}
	\mathscr F_t(G)
	=P(G;\mathring M)+\tanh t
	\int_S\operatorname{Tr}_S\chi_G\,dA_S
	=\tanh t\,P(G;\mathbb R^3)
	+(1-\tanh t)P(G;\mathring M).
	\label{eq:intro-coercivity}
\end{equation}
Thus $\mathscr F_t$ controls the full Euclidean perimeter independently of
the uncontrolled geometry. The direct method gives an
$\mathscr F_t$-minimizer in the constrained region, and Euclidean
isoperimetry shows that every such minimizer is bounded.

For each $t\in(0,t_0)$, the smallest $\mathscr F_t$-minimizer $G_t$ determines
a compact strongly stable minimal capillary separator $\Sigma_t$ whose
components are disks. Set
\[
E_t=\Omega(\Gamma_1)\setminus G_t.
\]
Boundedness and $\mathscr F_t$-minimality place every contact curve on the
boundary of the distinguished exterior and make $E_t$ strictly outer
$\mathfrak J_t$-minimizing. We then complete the compact side of $\Sigma_t$
away from $\mathcal E_+$, preserving the distinguished exterior and its mass. On the resulting
one-ended support, $\Sigma_t$ is strictly outer $J_t$-minimizing, so the
monotonicity and asymptotic results of \cite{EK} give
\[
m_+(S)\geq\frac1{\cosh t}\sqrt{\frac{|\Sigma_t|}{\pi}}
\geq\frac1{\cosh t}\sqrt{\frac{A_S}{\pi}}.
\]
Letting $t\downarrow0$ proves the inequality without requiring convergence
of $\Sigma_t$. For rigidity, lattice comparison orders the sets $G_t$.
Equality in the positive-parameter monotonicity then identifies the
corresponding portions of $S$ as nested pieces of one catenoid whose waist
disk realizes $A_S$.

Section~\ref{sec:outer-barrier} constructs the exterior calibration and
capillary separators. Section~\ref{sec:outward-theory} applies the one-ended
theory, and Section~\ref{sec:conditional-rigidity} proves rigidity and
sharpness.

\subsection*{Acknowledgments}

The author would like to thank Chao Xia and Jintian Zhu for discussions that inspired the problem considered in this paper.
This work was partially supported by NSFC (Grant No.~12501274).

\section{Preliminaries}
\label{sec:outer-barrier}

Throughout this section, assume $A_S>0$. Then $q\geq1$.

We first construct the minimal-disk foliation that provides the exterior calibration toward $\mathcal E_+$.

\begin{lemma}
	\label{lem:finite-slab-calibration}
	There exist admissible minimal disks $\Gamma_0\prec\Gamma_1$, constants
	$0<t_0<t_1$, and a smooth function $v$ on the closure of the component of
	$M\setminus\Gamma_0$ containing $\mathcal E_+$. The function satisfies
	$v\to\infty$ along $\mathcal E_+$, its level sets are minimal disks, and
	\[
	 \Gamma_i=\{v=t_i\},\qquad i=0,1.
	\]
	On the component of $M\setminus\Gamma_0$ containing $\mathcal E_+$,
	\begin{equation}
	\operatorname{div}\left(
	\frac{\bar\nabla v}{|\bar\nabla v|}
	\right)=0,
	\label{eq:calibration-interior}
	\end{equation}
	and on its lateral support boundary,
	\begin{equation}
	-\frac{\bar\nabla v}{|\bar\nabla v|}\cdot\nu_S>\tanh t_0.
	\label{eq:calibration-wall}
	\end{equation}
	Moreover, $v$ and its minimal level-set foliation extend smoothly across $\Gamma_0$ to the side away from $\mathcal E_+$.
\end{lemma}

\begin{proof}
	\emph{Step 1.} We first construct an admissible absolutely area-minimizing disk with boundary on a large coordinate sphere in $\mathcal E_+$.
	For a sufficiently large regular value $\lambda$,
	Lemma~\ref{lem:asymptotic-plateau-disk} gives a unique absolutely
	area-minimizing admissible disk $D_*$ with
	\[
	\partial D_*=S_+\cap\partial B_\lambda.
	\]

	\medskip
	\emph{Step 2.} We next smooth the mean-convex corner along $\partial D_*$ to obtain a one-ended asymptotically flat support surface.
	Along $\partial D_*$, the component of $M\setminus D_*$ containing $\mathcal E_+$ has the corner $D_*\cup S_+(D_*)$, with
	\[
	H_{D_*}=0,\qquad H_{S_+(D_*)}\geq0,\qquad
	|\nu_{D_*}\wedge\nu_S|>0
	\quad\text{on }\partial D_* .
	\]
	Thus the corner is transverse and mean-convex. Lemma~\ref{lem:localized-mean-convex-rounding} smooths this corner to a proper asymptotically flat support surface $\widehat S$ satisfying
	\[
	H_{\widehat S}\geq0,\qquad
	\int_{\widehat S}|H_{\widehat S}|\,dA<\infty .
	\]
	The modification is compact. By \eqref{eq:topology}, the resulting exterior region is diffeomorphic to the Euclidean half-space.

	\medskip
	\emph{Step 3.} We finally apply \cite[Proposition~13]{EK} to $\widehat S$. It gives a minimal-disk foliation. After a monotone reparametrization, the foliation is described by a smooth function $v$ satisfying
	\[
	v\to\infty,\qquad |\bar\nabla v|>0,\qquad
	H_{\{v=s\}}=0,
	\qquad
	-\frac{\bar\nabla v}{|\bar\nabla v|}
	\cdot\nu_{\widehat S}>\tanh v .
	\]

	Choose regular values $0<t_0<t_1$ sufficiently large that
	$\widehat S=S$ on a neighborhood of $\{v\geq t_0\}$, and define
\[
 \Gamma_i=\{v=t_i\},\qquad i=0,1.
\]
Then $\Gamma_0\prec\Gamma_1$ and the foliation transfers to $(M,S)$. On the component of $M\setminus\Gamma_0$ containing $\mathcal E_+$,
\[
\operatorname{div}\left(\frac{\bar\nabla v}{|\bar\nabla v|}\right)=H_{\{v=s\}}=0,
\]
and on $S\cap\{v\geq t_0\}$,
\[
-\frac{\bar\nabla v}{|\bar\nabla v|}\cdot\nu_S>\tanh v\geq\tanh t_0 .
\]
Since $\widehat S=S$ on a neighborhood of $\{v\geq t_0\}$, the function $v$ and its minimal level-set foliation extend across $\Gamma_0$ to the side away from $\mathcal E_+$.
\end{proof}

Let $\Gamma_0\prec\Gamma_1$, $t_0$, and $v$ be as in
Lemma~\ref{lem:finite-slab-calibration}, and set
\[
 K_0=\overline{\Omega(\Gamma_1)\setminus\Omega(\Gamma_0)}.
\]
For $t\in(0,t_0)$, define the functional $\mathscr F_t$ on a finite-perimeter set $G\subset\Omega(\Gamma_1)$ by
\begin{equation}
 \mathscr F_t(G)=P(G;\mathring M)
 +\tanh t\int_S\operatorname{Tr}_S\chi_G\,dA_S.
 \label{eq:barrier-free-energy}
\end{equation}
Fix an admissible reference disk $B\prec\Gamma_0$. Define the renormalized functional $\mathfrak J_t$ on an inner end region $E$ by
\begin{equation}
 \mathfrak J_t(E)=P(E;\mathring M)
 -\tanh t\int_S\bigl(\operatorname{Tr}_S\chi_E
 -\operatorname{Tr}_S\chi_{\Omega(B)}\bigr)\,dA_S.
 \label{eq:renormalized-J}
\end{equation}
For every inner end region $E$, the trace difference in \eqref{eq:renormalized-J} is supported on a compact portion of $S$, so the integral is finite. Changing the reference disk adds a constant independent of $E$.

\begin{proposition}
\label{prop:localized-outer-barrier}
For every $t\in(0,t_0)$, the following assertions hold.
\begin{enumerate}
 \item[(i)] If $G\subset\Omega(\Gamma_1)$ has finite perimeter and finite
 measure and contains a one-sided neighborhood of $\Gamma_1$ in
 $\Omega(\Gamma_1)$, then adjoining $K_0$ does not increase $\mathscr F_t$:
 \begin{equation}
  \mathscr F_t(G\cup K_0)\leq\mathscr F_t(G).
  \label{eq:outer-collar-replacement}
 \end{equation}
 Equality holds if and only if $K_0\subset G$ almost everywhere.

 \item[(ii)] If $F$ is an inner end region, then truncation at
 either fixed leaf does not increase $\mathfrak J_t$:
 \begin{equation}
  \mathfrak J_t\bigl(F\cap\Omega(\Gamma_i)\bigr)
  \leq\mathfrak J_t(F),\qquad i=0,1.
  \label{eq:global-outer-localization}
 \end{equation}
 For each $i$, equality holds if and only if
 $F\setminus\Omega(\Gamma_i)$ has measure zero.

 \item[(iii)] Every $\mathscr F_t$-minimizer $G$ among finite-perimeter sets
 $G'$ satisfying
 \[
  K_0\subset G'\subset\Omega(\Gamma_1),
  \qquad |G'|<\infty,
 \]
 contains a slab extending across $\Gamma_0$: there is a minimal leaf
 $\Gamma_-\prec\Gamma_0$ in the extension across $\Gamma_0$ such that
 \[
  \Omega(\Gamma_1)\setminus\Omega(\Gamma_-)\subset G
  \quad\text{almost everywhere}.
 \]
\end{enumerate}
\end{proposition}

\begin{proof}
\emph{(i).}
For a smooth set transverse to the leaves and to $S$, the divergence theorem
and \eqref{eq:calibration-interior} give
\begin{align}
 0
 &=\int_{K_0\setminus G}
 \operatorname{div}\left(\frac{\bar\nabla v}{|\bar\nabla v|}\right)
 \notag\\
 &=-\int_{\partial^*G\cap\mathring K_0}
 \frac{\bar\nabla v}{|\bar\nabla v|}\cdot\nu_G\,d\mathcal H^2
 +\int_{S\cap(K_0\setminus G)}
 \frac{\bar\nabla v}{|\bar\nabla v|}\cdot\nu_S\,dA_S -\mathcal H^2\bigl(\Gamma_0\cap\partial(K_0\setminus G)\bigr)
\end{align}
Then
\begin{align}
 &\mathscr F_t(G)-\mathscr F_t(G\cup K_0)
 \\=&P(G;\mathring M)-P(G\cup K_0;\mathring M) +\tanh t\int_S
 \left(\operatorname{Tr}_S\chi_G
 -\operatorname{Tr}_S\chi_{G\cup K_0}\right)dA_S
 \notag\\
 \geq&
 \mathcal H^2(\partial^*G\cap\mathring K_0)
 -\mathcal H^2\bigl(\Gamma_0\cap\partial(K_0\setminus G)\bigr) -\tanh t\,
 \mathcal H^2\bigl(S\cap(K_0\setminus G)\bigr)
 \notag\\
 =&
 \int_{\partial^*G\cap\mathring K_0}
 \left(1+\frac{\bar\nabla v}{|\bar\nabla v|}
 \cdot\nu_G\right)\,d\mathcal H^2
 +\int_{S\cap K_0}
 \left(-\frac{\bar\nabla v}{|\bar\nabla v|}
 \cdot\nu_S-\tanh t\right)
 \operatorname{Tr}_S\chi_{K_0\setminus G}\,dA_S\geq 0.
 \label{eq:outer-calibration-defect}
\end{align}
Indeed,
\[  -\frac{\bar\nabla v}{|\bar\nabla v|}\cdot\nu_S-\tanh t
 >\tanh t_0-\tanh t>0
\]
by \eqref{eq:calibration-wall}, and $0<t<t_0$. For an arbitrary
finite-perimeter set, apply the Gauss--Green formula directly to
$K_0\setminus G$. Together with the BV perimeter formula for unions, it gives
the same inequality with reduced boundaries and traces.
Thus \eqref{eq:outer-calibration-defect} and its two nonnegative defect terms
remain valid; compare \cite[Lemma~16]{EK}.

If equality holds, then
\[
 \frac{\bar\nabla v}{|\bar\nabla v|}\cdot\nu_G=-1
 \quad\mathcal H^2\text{-a.e. on }
 \partial^*G\cap\mathring K_0,
 \qquad
 \operatorname{Tr}_S\chi_{K_0\setminus G}=0
 \quad\mathcal H^2\text{-a.e. on }S\cap K_0.
\]
Thus $\partial^*(K_0\setminus G)$ is calibrated by
$\bar\nabla v/|\bar\nabla v|$. The BV constancy theorem for the
zero-mean-curvature foliation implies that any nonzero component of
$K_0\setminus G$ has nonzero lateral trace. The second equality therefore gives
$|K_0\setminus G|=0$.  The converse is immediate.

\medskip
\emph{(ii).}
Fix $i\in\{0,1\}$.  Since $F$ excludes a full tail of $\mathcal E_+$,
$F\setminus\Omega(\Gamma_i)$ is bounded.  The same direct Gauss--Green
calculation gives
\begin{align}
 \mathfrak J_t(F)
 -\mathfrak J_t\bigl(F\cap\Omega(\Gamma_i)\bigr)
 \geq{}&\int_{\partial^*F\setminus\Omega(\Gamma_i)}
 \left(1-\frac{\bar\nabla v}{|\bar\nabla v|}
 \cdot\nu_F\right)\,d\mathcal H^2\notag\\
 &+\int_{S_i^+}
 \left(-\frac{\bar\nabla v}{|\bar\nabla v|}
 \cdot\nu_S-\tanh t\right)
 \operatorname{Tr}_S\chi_{F\setminus\Omega(\Gamma_i)}\,dA_S,
 \label{eq:global-calibration-defect}
\end{align}
Here $S_i^+$ is the component of $S\setminus\partial\Gamma_i$ containing
$\mathcal E_+$.
 Both integrands are nonnegative, and the second is strictly
positive wherever the trace of $F\setminus\Omega(\Gamma_i)$ is nonzero, by
\eqref{eq:calibration-wall}.

If equality holds, then
\[
 \frac{\bar\nabla v}{|\bar\nabla v|}\cdot\nu_F=1
 \quad\mathcal H^2\text{-a.e. on }
 \partial^*F\setminus\Omega(\Gamma_i),
 \qquad
 \operatorname{Tr}_S\chi_{F\setminus\Omega(\Gamma_i)}=0
 \quad\mathcal H^2\text{-a.e. on }S_i^+.
\]
As in (i), combine the BV constancy theorem for the zero-mean-curvature
foliation with boundedness and the zero lateral trace. This gives
$|F\setminus\Omega(\Gamma_i)|=0$.  The converse is immediate.

\medskip
\emph{(iii).}
By Lemma~\ref{lem:finite-slab-calibration}, the zero-mean-curvature foliation
extends smoothly to the inner side of $\Gamma_0$. Since $t<t_0$ and
\[
 -\frac{\bar\nabla v}{|\bar\nabla v|}\cdot\nu_S>\tanh t_0
\]
on the lateral support at and outside $\Gamma_0$, compactness and continuity
give an inner leaf $\Gamma_-\prec\Gamma_0$ sufficiently close to
$\Gamma_0$. Set
\[
 K_-=\overline{\Omega(\Gamma_1)\setminus\Omega(\Gamma_-)}.
\]
Throughout $K_-$,
\[
 \operatorname{div}\!\left(
 \frac{\bar\nabla v}{|\bar\nabla v|}
 \right)=0,
 \qquad
 -\frac{\bar\nabla v}{|\bar\nabla v|}\cdot\nu_S>\tanh t.
\]
Applying the comparison from (i) to $K_-\setminus G$ gives
\[
 \mathscr F_t(G\cup K_-)\leq\mathscr F_t(G).
\]
The set $G\cup K_-$ is admissible, so the $\mathscr F_t$-minimality of $G$
forces equality.
The same equality argument applied to $K_-$ gives $|K_-\setminus G|=0$.
Thus (iii) follows.
\end{proof}

We now use the exterior barrier to minimize $\mathscr F_t$. Fix $t\in(0,t_0)$, and let $E$ be an inner end region. To avoid the loss of compactness caused by the negative trace term in \eqref{eq:renormalized-J}, pass to its complement in $\Omega(\Gamma_1)$:
\begin{equation}
 G=\Omega(\Gamma_1)\setminus E.
\end{equation}

If $G$ contains a relative neighborhood of $\Gamma_1$ in
$\Omega(\Gamma_1)$, then
 \begin{align}
 \mathscr F_t(G) =\tanh t\,P(G;\mathbb R^3)
 +(1-\tanh t)P(G;\mathring M) \label{eq:full-perimeter-identity} .
 \end{align}

Define the admissible class
\begin{equation}
 \mathscr G(\Gamma_0,\Gamma_1)=
 \left\{G\subset\Omega(\Gamma_1):
 G\text{ has locally finite perimeter},
 K_0\subset G\text{ a.e.},\quad |G|<\infty
 \right\}.
 \label{eq:admissible-G-class}
\end{equation}
No boundedness assumption is imposed on $G$.

\begin{proposition}
\label{prop:canonical-F-minimizer}
For every $t\in(0,t_0)$, the functional $\mathscr F_t$ admits a unique
smallest $\mathscr F_t$-minimizer $G_t$ in $\mathscr G(\Gamma_0,\Gamma_1)$.
Every $\mathscr F_t$-minimizer is bounded in $\mathbb R^3$. Moreover, if $G$
is any $\mathscr F_t$-minimizer, then
$G_t\subset G$ almost everywhere.
\end{proposition}
\begin{proof}
Let $\{G_j\}$ be a minimizing sequence of $\mathscr F_t$. Then
$\sup_j\mathscr F_t(G_j)<\infty$. Moreover,
\begin{equation}
 \sup_j\left(P(G_j;\mathbb R^3)+|G_j|\right)<\infty.
 \label{eq:global-bv-bound}
\end{equation}
Here we used the Euclidean isoperimetric inequality and
\eqref{eq:full-perimeter-identity}, which give
\[
 |G_j|^{2/3}\leq C_{\mathrm{iso}}P(G_j;\mathbb R^3)
 \leq C_{\mathrm{iso}}(\tanh t)^{-1}\mathscr F_t(G_j).
\]
For each $R>0$, BV compactness gives a subsequence converging in
$L^1(B_R)$.  A diagonal argument yields a measurable set $G$ such that
\[
 \chi_{G_j}\longrightarrow\chi_G
 \quad\text{in }L^1_{\mathrm{loc}}(\mathbb R^3).
\]
Because $K_0\subset G_j$ almost everywhere, the same inclusion holds for
$G$. Fatou's lemma and \eqref{eq:global-bv-bound} give $|G|<\infty$.
The condition $G\subset\Omega(\Gamma_1)$ also passes to the limit.

For every ball $B_R$, lower semicontinuity gives
\begin{align*}
 &\tanh t\,P(G;B_R)+(1-\tanh t)P(G;\mathring M\cap B_R)\\
 \le&\liminf_{j\to\infty}
 \left(\tanh t\,P(G_j;B_R)
 +(1-\tanh t)P(G_j;\mathring M\cap B_R)\right).
\end{align*}
Letting $R\to\infty$  yields
\[
 \mathscr F_t(G)\leq\liminf_{j\to\infty}\mathscr F_t(G_j).
\]
Thus $G\in\mathscr G(\Gamma_0,\Gamma_1)$ is an $\mathscr F_t$-minimizer.

Next, we show that every $\mathscr F_t$-minimizer is bounded. Let
$G$ be an $\mathscr F_t$-minimizer and choose $R_0$ with
$K_0\subset B_{R_0}$. For
$r>R_0$, set
\[
 V(r)=|G\setminus B_r|,\qquad
 A(r)=\mathcal H^2(G\cap\partial B_r).
\]
For almost every $r$, the coarea theorem gives
\begin{equation}
 -V'(r)=A(r),
 \label{eq:tail-coarea}
\end{equation}
and $G\cap B_r$ is a finite-perimeter member of
$\mathscr G(\Gamma_0,\Gamma_1)$.

Direct computation gives
\begin{align}
 0
 &\leq \mathscr F_t(G\cap B_r)-\mathscr F_t(G)\notag\\
 &\leq \tanh t\bigl(P(G;B_r)+A(r)-P(G;\mathbb R^3)\bigr)\notag\\
 &\quad +(1-\tanh t)
 \bigl(P(G;\mathring M\cap B_r)+A(r)-P(G;\mathring M)\bigr)\notag\\
 &=A(r)-\tanh t\,|D\chi_G|
 (\mathbb R^3\setminus\overline{B_r})
 -(1-\tanh t)|D\chi_G|
 (\mathring M\setminus\overline{B_r}),
 \label{eq:weighted-tail-bound}
\end{align}
which implies
\begin{align}
P(G;\mathbb R^3\setminus\overline{B_r})\le(\tanh t)^{-1} A(r).
\end{align}
Consequently,
\begin{align}
 P(G\setminus B_r;\mathbb R^3)
 &\leq P(G;\mathbb R^3\setminus\overline{B_r})+A(r)\notag\\
 &\leq\bigl(1+(\tanh t)^{-1}\bigr)A(r).
 \label{eq:tail-perimeter}
\end{align}

Set $C_t=C_{\mathrm{iso}}\bigl(1+(\tanh t)^{-1}\bigr)$.  The Euclidean
isoperimetric inequality and \eqref{eq:tail-coarea} give
\begin{align}
 V(r)^{2/3}
 &=|G\setminus B_r|^{2/3}
  \leq C_{\mathrm{iso}}P(G\setminus B_r;\mathbb R^3)
  \leq C_tA(r)=-C_tV'(r).
 \label{eq:finite-extinction-ode}
\end{align}
for almost every $r>R_0$ with $V(r)>0$.  Hence
\begin{align*}
 V(r)^{1/3}
 &\leq V(R_0)^{1/3}-\frac{r-R_0}{3C_t}
\end{align*}
as long as $V(r)>0$.  Thus $G$ is bounded.

By
\begin{equation}
 \mathscr F_t(G\cap H)+\mathscr F_t(G\cup H)
 \leq\mathscr F_t(G)+\mathscr F_t(H),
 \label{eq:F-lattice}
\end{equation}
the intersection and union of two $\mathscr F_t$-minimizers are again
$\mathscr F_t$-minimizers. The same direct-method argument shows that their
volumes attain an infimum. Let $G_t$ be a volume minimizer among the
$\mathscr F_t$-minimizers.

If $G$ is any other $\mathscr F_t$-minimizer, then $G_t\cap G$ is an
$\mathscr F_t$-minimizer.  Since $G_t$ is a volume minimizer,
$|G_t\cap G|\geq|G_t|$,
and hence $G_t\subset G$ almost everywhere.  This proves both the smallest
property and uniqueness.
\end{proof}

Set
\begin{equation}
 E_t=\Omega(\Gamma_1)\setminus G_t,\qquad
 \Sigma_t=
 \overline{\partial^*E_t\cap\mathring M}\setminus\Gamma_1.
 \label{eq:Et-Sigmat}
\end{equation}

\begin{figure}[htbp]
\centering
\begin{tikzpicture}[x=1.45cm,y=1.25cm]
 \begin{scope}
  \clip (0,2.20) .. controls (2.0,2.34) and (4.0,2.08) .. (7.4,2.20)
  -- (7.4,0.52) .. controls (4.0,0.42) and (2.0,0.30) .. (0,0.50) -- cycle;
  \fill[gray!12] (-0.1,0.2) rectangle (3.0,2.5);
  \fill[gray!30] (3.0,0.2) rectangle (6.55,2.5);
  \fill[gray!5] (6.55,0.2) rectangle (7.5,2.5);
 \end{scope}
 \draw[very thick]
 (0,2.20) .. controls (2.0,2.34) and (4.0,2.08) .. (7.4,2.20);
 \draw[very thick]
 (0,0.50) .. controls (2.0,0.30) and (4.0,0.42) .. (7.4,0.52);
 \node[above] at (0.8,2.24) {$S$};

 \draw[thick,dashed]
 (3.0,0.39) .. controls (2.86,0.90) and (2.86,1.72) .. (3.0,2.20);
 \node[below] at (3.0,0.36) {$\Sigma_t$};

 \draw[very thick]
 (5.20,0.46) .. controls (5.08,0.98) and (5.08,1.65) .. (5.20,2.13);
 \node[above] at (5.20,2.14) {$\Gamma_0$};
 \draw[very thick]
 (6.55,0.48) .. controls (6.43,0.98) and (6.43,1.68) .. (6.55,2.16);
 \node[above] at (6.55,2.17) {$\Gamma_1$};

 \node at (1.45,1.36) {$E_t$};
 \node at (4.08,1.36) {$G_t$};
 \draw[thin] (5.20,0.16) -- (5.20,0.04) -- (6.55,0.04) -- (6.55,0.16);
 \node[below] at (5.88,0.04) {$K_0\subset G_t$};

 \draw[->] (1.05,0.94) -- (0.08,0.94);
 \node[below] at (0.55,0.94) {other ends};
 \draw[->] (6.85,1.35) -- (7.55,1.35);
 \node[below] at (7.15,1.35) {$\mathcal E_+$};
 \draw[->] (5.36,2.58) -- (6.39,2.58);
 \node[above] at (5.88,2.58) {$\Gamma_0\prec\Gamma_1$};
\end{tikzpicture}
\caption{The sets $E_t$ and $G_t$ in $\Omega(\Gamma_1)$.
The bounded set $G_t$ contains the fixed slab $K_0$, while the inner end region
$E_t$ contains the remaining ends.}
\end{figure}
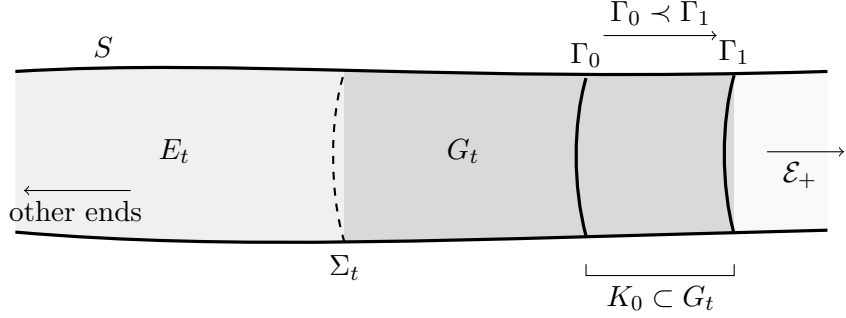

\begin{proposition}
\label{prop:capillary-separator}
There are fixed admissible disks $\Gamma_0\prec\Gamma_1$ and $t_0>0$ with
the following property. For every $t\in(0,t_0)$, the smallest
$\mathscr F_t$-minimizer $G_t$ determines a nonempty compact smooth properly
embedded strongly stable minimal capillary separator
$\Sigma_t\in\mathcal A_S$. Every component of $\Sigma_t$ is a disk, and
\[
 \langle\nu_{\Sigma_t},\nu_S\rangle=-\tanh t.
\]
Let $S_t^+$ be the component of $S\setminus\partial\Sigma_t$ containing $\mathcal E_+$. Then $\mathcal E_+$ is its only
end, and the components of $\partial S_t^+$ are precisely the contact curves
of $\Sigma_t$.  The separators satisfy
\begin{equation}
 A_S\leq|\Sigma_t|
 \leq|\Gamma_0|
 +\tanh t\,|S\cap(\Omega(\Gamma_1)\setminus\Omega(\Gamma_0))|,
 \label{eq:capillary-area-bounds}
\end{equation}
and $E_t$ is strictly outer $\mathfrak J_t$-minimizing: every inner end region $F$
enclosing $E_t$ satisfies
\begin{equation}
 \mathfrak J_t(F)\geq\mathfrak J_t(E_t),
 \label{eq:strict-outer-J}
\end{equation}
with equality if and only if $F=E_t$ almost everywhere.
\end{proposition}

\begin{proof}
We prove the assertions in four steps.

\medskip
\emph{Step 1.}
We first show that $E_t$ is an inner end region and that $\Sigma_t$
stays a positive distance from the fixed faces $\Gamma_0$ and $\Gamma_1$.
We then prove regularity and nonemptiness. Since $G_t$ is bounded, $E_t$ contains a full tail of
each uncontrolled end $\mathcal E_j$, $1\leq j\leq q$, and hence is an inner end region.

Proposition~\ref{prop:localized-outer-barrier}(iii) gives a fixed leaf
$\Gamma_-\prec\Gamma_0$ such that the enlarged slab between $\Gamma_-$ and
$\Gamma_1$ lies in $G_t$ almost everywhere. Hence the reduced boundary of
$G_t$ does not meet the interior of this slab. In particular,
\[
\operatorname{dist}\bigl(\Sigma_t,\Gamma_0\cup\Gamma_1\bigr)>0.
\]

We now establish the asserted regularity. Boundedness of $G_t$ and the
preceding separation show that $\Sigma_t$ is compact and
$\partial\Sigma_t\subset S$. Apply Taylor's regularity theorem for the
$\mathscr F_t$-minimizing capillarity set $G_t$ \cite{Taylor}, as in
\cite[proof of Proposition~17]{EK}. It shows that $\Sigma_t$ is smoothly
embedded up to $S$ and has no singular points. Compactness makes the
embedding proper.

The surface $\Sigma_t$ is nonempty. Otherwise, BV constancy in the connected
region $\Omega(\Gamma_1)$ would give either $G_t=\varnothing$ or
$G_t=\Omega(\Gamma_1)$ almost everywhere. The former contradicts
$K_0\subset G_t$, and the latter contradicts the boundedness of $G_t$.
Finally, every component of $\Sigma_t$ has nonempty boundary. Indeed, the
interior first variation would otherwise produce a compact minimal surface
in $\mathbb R^3$, which is impossible.

\medskip
 \emph{Step 2.}
We next derive zero mean curvature, the contact angle, and strong stability.
We also determine the topology of each component. The $\mathscr F_t$-minimality of $G_t$ and
Lemma~\ref{lem:complement-first-variation} give
\[
 H_{\Sigma_t}=0,
 \qquad
 \langle\nu_{\Sigma_t},\nu_S\rangle=-\tanh t.
\]
Every sufficiently small admissible variation remains in
$\mathscr G(\Gamma_0,\Gamma_1)$. Hence the $\mathscr F_t$-minimality of $G_t$
yields the strong capillary stability inequality in \cite[Appendix~C]{EK}.
Thus every component of $\Sigma_t$ is a two-sided stable minimal capillary
surface with nonempty boundary. Since $H_S\geq0$ along its boundary,
\cite[Lemma~49]{EK} implies that it is a disk. Compactness of $\Sigma_t$
gives finitely many components.

\medskip
 \emph{Step 3.}
We first determine the region adjacent to $\mathcal E_+$ and then prove the area bounds in
\eqref{eq:capillary-area-bounds}.  Define
\[
G_t^+:=G_t\cup\bigl(\overline M\setminus\Omega(\Gamma_1)\bigr).
\]
We use the regular open representative of $G_t^+$ determined by its smooth
interface. Then
$\Sigma_t=\overline{\partial G_t^+\cap\mathring M}$.

By \eqref{eq:topology}, compactifying the ends gives a $3$-ball. Cut this
ball along the properly embedded disk components of $\Sigma_t$. The intersection with $M$ of the region containing $p_+$ is a component of $G_t^+$ and contains $K_0$. It is the only component of $G_t^+$. Indeed, any other component would determine a
component of $G_t$ disjoint from $K_0$. Removing that component preserves
admissibility and strictly decreases the perimeter term. It cannot increase
the nonnegative trace term in $\mathscr F_t$, so this contradicts the
$\mathscr F_t$-minimality of $G_t$.

Each disk component of $\Sigma_t$ separates the region containing $p_+$ from a component of $E_t$. Successive separation by these disks shows that the incident part of
the support is connected and has every contact curve as a boundary component.
This part is $S_t^+$. Since $G_t$ is bounded, $E_t$ contains a full tail of
every remaining end. Therefore $\mathcal E_+$ is the only end of $S_t^+$ and
$\partial S_t^+=\partial\Sigma_t$.

We now prove the area bounds in \eqref{eq:capillary-area-bounds}. The surface
$\Sigma_t$ separates $\mathcal E_+$ from the remaining ends and is therefore
admissible. The definition of $A_S$ gives
\[
A_S\leq|\Sigma_t|.
\]
For the upper bound, comparison with $K_0$ gives
\begin{align*}
	|\Gamma_1|+|\Sigma_t|
	&=P(G_t;\mathring M)\leq\mathscr F_t(G_t)\leq\mathscr F_t(K_0)\\
	&=P(K_0;\mathring M)+\tanh t\int_S\operatorname{Tr}_S\chi_{K_0}\,dA_S\\
	&=|\Gamma_1|+|\Gamma_0|+\tanh t\,|S\cap(\Omega(\Gamma_1)\setminus\Omega(\Gamma_0))|.
\end{align*}
Cancelling $|\Gamma_1|$ gives the upper bound in
\eqref{eq:capillary-area-bounds}.

\medskip
\emph{Step 4.}
Finally, we prove strict outer $\mathfrak J_t$-minimality. Let $F$ be an inner end region such that $E_t\subset F$ almost everywhere.
Proposition~\ref{prop:localized-outer-barrier}(ii) gives
\[
 \mathfrak J_t\bigl(F\cap\Omega(\Gamma_0)\bigr)\leq\mathfrak J_t(F),
\]
with equality if and only if $F\setminus\Omega(\Gamma_0)$ has measure zero.
Since $K_0\subset G_t$, we have $E_t\subset\Omega(\Gamma_0)$.  Together with $E_t\subset F$, this gives
\[
E_t\subset F\cap\Omega(\Gamma_0)
\]
almost everywhere.

Passing to complements, set
\[
G_F=\Omega(\Gamma_1)\setminus\bigl(F\cap\Omega(\Gamma_0)\bigr).
\]
Then
\[
K_0\subset G_F\subset G_t,
\qquad
G_F\in\mathscr G(\Gamma_0,\Gamma_1)
\]
almost everywhere.  Since both $G_F$ and $G_t$ contain a collar of
$\Gamma_1$, using the $\mathscr F_t$-minimality of $G_t$
gives
\begin{align*}
	\mathfrak J_t(F)-\mathfrak J_t(E_t)
	&\geq\mathfrak J_t\bigl(F\cap\Omega(\Gamma_0)\bigr)-\mathfrak J_t(E_t)\\
	&=P\bigl(F\cap\Omega(\Gamma_0);\mathring M\bigr)-P(E_t;\mathring M)
	-\tanh t\int_S\left(\operatorname{Tr}_S\chi_{F\cap\Omega(\Gamma_0)}
	-\operatorname{Tr}_S\chi_{E_t}\right)dA_S\\
	&=P(G_F;\mathring M)-P(G_t;\mathring M)
	+\tanh t\int_S\left(\operatorname{Tr}_S\chi_{G_F}
	-\operatorname{Tr}_S\chi_{G_t}\right)dA_S\\
	&=\mathscr F_t(G_F)-\mathscr F_t(G_t)\geq0.
\end{align*}
Suppose equality holds. Then $F=F\cap\Omega(\Gamma_0)$ almost everywhere,
and $G_F$ is an $\mathscr F_t$-minimizer. Proposition~\ref{prop:canonical-F-minimizer}
gives $G_t\subset G_F$. On the other hand,
$E_t\subset F\cap\Omega(\Gamma_0)$ gives $G_F\subset G_t$. Hence $G_F=G_t$.
It follows that $F\cap\Omega(\Gamma_0)=E_t$ and therefore $F=E_t$ almost
everywhere.
\end{proof}

\section{Proof of the inequality}
\label{sec:outward-theory}

Throughout this section, assume $A_S>0$.

\begin{theorem}
\label{thm:capillary-mass}
For every $t\in(0,t_0)$, one has
\begin{equation}
 m_+(S)\geq\frac1{\cosh t}
 \sqrt{\frac{|\Sigma_t|}{\pi}}.
 \label{eq:capillary-mass}
\end{equation}
\end{theorem}

\begin{proof}

\emph{Step 1.} We first complete the compact side of $\Sigma_t$
without changing the distinguished exterior, the capillary data, or the mass.

Write $\Sigma_t=D_t^1\cup\cdots\cup D_t^k$. By
Proposition~\ref{prop:capillary-separator}, each $D_t^i$ is a disk,
$\partial S_t^+=\bigcup_{i=1}^k\partial D_t^i$, and $\mathcal E_+$ is the
only end of $S_t^+$. For each $i$, extend $S_t^+$ across $\partial D_t^i$
by a short collar on the side away from $\mathcal E_+$. Cap the new boundary
curve by pushing $D_t^i$ slightly away from the preserved exterior. The
collars and pushed-off disks can be chosen pairwise disjoint. Smoothing their
junctions away from $\partial\Sigma_t$ produces a smooth embedded support
$\widetilde S_t$. All modifications lie on the side of $\Sigma_t$ away from
$\mathcal E_+$. Thus the new support agrees with $S$ near $\Sigma_t$ and
throughout the exterior toward $\mathcal E_+$.

The surface $S_t^+$ is planar, has the single end $\mathcal E_+$, and has
finitely many boundary components. Capping them produces a properly embedded
topological plane $\widetilde S_t$. Outside a compact set, this plane is a
graph. It therefore separates $\mathbb R^3$ into two components whose
closures are topological half-spaces. Let $\widetilde M_t$ be the component
containing the preserved exterior.

The construction leaves $\Sigma_t$ unchanged, and $\widetilde S_t$ agrees
with $S$ near $\partial\Sigma_t$. Hence $\Sigma_t$ remains a strongly stable
minimal capillary surface with parameter $t$ in
$(\widetilde M_t,\widetilde S_t)$. The component of
$\widetilde S_t\setminus\partial\Sigma_t$ containing $\mathcal E_+$ is the
unchanged surface $S_t^+$. It agrees with $S_+$ outside a compact set, while
the rest of $\widetilde S_t$ is smooth and compact. Thus
\eqref{eq:natural-hypotheses} gives $H_{\widetilde S_t}\geq0$ on this
component and $H_{\widetilde S_t}\in L^1(\widetilde S_t)$. The same
agreement at infinity gives
\[
m(\widetilde S_t)=m_+(S).
\]

\medskip
\emph{Step 2.} We next transfer strict outer $\mathfrak J_t$-minimality to strict outer $J_t$-minimality in $(\widetilde M_t,\widetilde S_t)$.

It suffices to compare the completed free energy with $\mathfrak J_t$. Let
$J_t$ denote the functional in \eqref{eq:EK-free-energy} for the completed
support $\widetilde S_t$. The region of $\widetilde M_t$ separated from
$\mathcal E_+$ by $\Sigma_t$ is denoted by $\widetilde E_t$. It is the
compact-side replacement of the original inner end region $E_t$. Now let
$\Gamma\subset\overline{\widetilde M_t}$ be an admissible surface enclosing
$\Sigma_t$, and let $\widetilde E_\Gamma$ be the region it separates from
$\mathcal E_+$. Then
\[
\widetilde E_t\subset\widetilde E_\Gamma,
\qquad
(\widetilde M_t,\widetilde S_t)=(M,S)
\quad\text{on }
\overline{\widetilde E_\Gamma\setminus\widetilde E_t}.
\]
Consequently, $ E_t\cup\bigl(\widetilde E_\Gamma\setminus\widetilde E_t\bigr)$
is an inner end region in $M$ enclosing $E_t$. Direct computation gives
\begin{align*}
	J_t(\Gamma)-J_t(\Sigma_t)
	&=|\Gamma|-|\Sigma_t|
	-\tanh t\int_{\widetilde S_t}
	\left(
	\operatorname{Tr}\chi_{\widetilde E_\Gamma}
	-\operatorname{Tr}\chi_{\widetilde E_t}
	\right)dA_{\widetilde S_t}\\
	&=P\bigl(E_t\cup(\widetilde E_\Gamma\setminus\widetilde E_t);\mathring M\bigr)
	-P(E_t;\mathring M)\\
	&\quad-\tanh t\int_S
	\left(
	\operatorname{Tr}_S\chi_{E_t\cup(\widetilde E_\Gamma\setminus\widetilde E_t)}
	-\operatorname{Tr}_S\chi_{E_t}
	\right)dA_S\\
	&=\mathfrak J_t\bigl(E_t\cup(\widetilde E_\Gamma\setminus\widetilde E_t)\bigr)
	-\mathfrak J_t(E_t)\\
	&\geq0.
\end{align*}
The last inequality and its equality case follow from
Proposition~\ref{prop:capillary-separator}. Equality holds only if
$|\widetilde E_\Gamma\setminus\widetilde E_t|=0$. Thus $\Sigma_t$ is
strictly outer $J_t$-minimizing in the completion. This is exactly
condition~(30) of \cite{EK}.

\medskip
\emph{Step 3.} Apply the outward theory of \cite{EK}.
Every competitor $\Sigma$ in \cite[Proposition~17]{EK} satisfies
\[
\widetilde E_t\subset\Omega(\Sigma),
\qquad
\Sigma\subset\overline{\widetilde M_t\setminus\widetilde E_t},
\qquad
(\widetilde M_t,\widetilde S_t)=(M,S)
\quad\text{on }
\overline{\widetilde M_t\setminus\widetilde E_t}.
\]
Within $\overline{\widetilde M_t\setminus\widetilde E_t}$, Lemma~50 and
Proposition~13 of \cite{EK} provide the auxiliary constructions near
$\Sigma_t$ and at infinity, respectively. Every outward capillary surface
$\Sigma$ produced by the construction satisfies
\[
\partial\Sigma\subset S_t^+,
\qquad
H_{\widetilde S_t}\geq0\quad\text{on }S_t^+.
\]
Thus the mean-curvature term in \cite[Corollary~33]{EK} is evaluated only
where it is nonnegative. Combining that corollary with the asymptotic
identification in \cite[Proposition~37]{EK} gives
\[
m_+(S)=m(\widetilde S_t)
\geq\frac1{\cosh t}\sqrt{\frac{|\Sigma_t|}{\pi}},
\]
which proves \eqref{eq:capillary-mass}.
\end{proof}

\begin{proof}[Proof of Theorem~\ref{thm:main}]
If $A_S=0$, the positive mass theorem with arbitrary ends and noncompact
boundary \cite{PMArbitraryEnds} gives $m_+(S)\geq0$, with equality if and only
if $(M,g)$ is isometric to the Euclidean half-space. Since $g$ is the
Euclidean metric, the isometry is the restriction of a rigid motion. This is
equivalent to $q=0$ and the half-space description in the statement.
Assume $A_S>0$.  For every $t\in(0,t_0)$,
Proposition~\ref{prop:capillary-separator} and
Theorem~\ref{thm:capillary-mass} give
\[
 m_+(S)\geq\frac1{\cosh t}
 \sqrt{\frac{|\Sigma_t|}{\pi}}
 \geq\frac1{\cosh t}\sqrt{\frac{A_S}{\pi}}.
\]
Letting $t\downarrow0$ proves \eqref{eq:main-inequality}.
The equality statement is proved in
Theorem~\ref{thm:conditional-rigidity}.
\end{proof}

\section{Rigidity and sharpness}
\label{sec:conditional-rigidity}

Throughout this section, assume $A_S>0$.

We first order the canonical capillary separators and then analyze equality
at a positive capillary parameter.

\begin{lemma}
\label{lem:ordered-capillary-separators}
If $0<s<t<t_0$, then
\begin{equation}
 G_t\subset G_s,
 \qquad E_s\subset E_t
 \label{eq:ordered-capillary-regions}
\end{equation}
almost everywhere.  Moreover,
\begin{equation}
 \frac1{\cosh s}\sqrt{\frac{|\Sigma_s|}{\pi}}
 \leq
 \frac1{\cosh t}\sqrt{\frac{|\Sigma_t|}{\pi}}.
 \label{eq:ordered-capillary-masses}
\end{equation}
\end{lemma}

\begin{proof}
The BV perimeter inequality for unions and intersections, together with
$\tanh s<\tanh t$, gives
\[
\mathscr F_s(G_s\cup G_t)+\mathscr F_t(G_s\cap G_t)
\leq\mathscr F_s(G_s)+\mathscr F_t(G_t).
\]
The reverse inequality follows from $\mathscr F_s$- and
$\mathscr F_t$-minimality. Hence $G_s\cup G_t$ and $G_s\cap G_t$ are
$\mathscr F_s$- and $\mathscr F_t$-minimizers, respectively. Since $G_t$
is the smallest $\mathscr F_t$-minimizer,
$G_t=G_s\cap G_t$ almost everywhere. This proves
\eqref{eq:ordered-capillary-regions}.

Complete the compact side of $\Sigma_s$ as in the proof of
Theorem~\ref{thm:capillary-mass}. Denote the resulting pair by
$(\widetilde M_s,\widetilde S_s)$. By
\eqref{eq:ordered-capillary-regions}, $\Sigma_t\subset
\overline{\widetilde M_s}$ is admissible and encloses $\Sigma_s$. Let
$\Gamma\subset\overline{\widetilde M_s}$ be any admissible surface enclosing
$\Sigma_s$, and let $\widetilde E_s$, $\widetilde E_t$, and
$\widetilde E_\Gamma$ be the regions separated from $\mathcal E_+$ by
$\Sigma_s$, $\Sigma_t$, and $\Gamma$, respectively. As in Step~2 of the
proof of Theorem~\ref{thm:capillary-mass},
\begin{align*}
 F&=E_s\cup(\widetilde E_\Gamma\setminus\widetilde E_s)
 \quad\text{is an inner end region in }M,\\
 E_t&=E_s\cup(\widetilde E_t\setminus\widetilde E_s),
\end{align*}
\begin{align*}
 J_t(\Gamma)-J_t(\Sigma_t)
 &=|\Gamma|-|\Sigma_t|
 -\tanh t\int_{\widetilde S_s}
 \left(\operatorname{Tr}\chi_{\widetilde E_\Gamma}
 -\operatorname{Tr}\chi_{\widetilde E_t}\right)dA_{\widetilde S_s}\\
 &=\mathfrak J_t(F)-\mathfrak J_t(E_t)\\
 &\geq
 \mathfrak J_t\bigl(F\cap\Omega(\Gamma_0)\bigr)-\mathfrak J_t(E_t)\\
 &=\mathscr F_t\!\left(
 \Omega(\Gamma_1)\setminus\bigl(F\cap\Omega(\Gamma_0)\bigr)
 \right)-\mathscr F_t(G_t)\geq0.
\end{align*}
Thus $\Sigma_t$ is a $J_t$-minimizer in
$(\widetilde M_s,\widetilde S_s)$ among admissible surfaces enclosing
$\Sigma_s$.

Let $\Sigma_t^-\subset\overline{\widetilde M_s}$ be a $J_t$-minimizer
enclosing $\Sigma_s$ with least lateral support area.  Since $\Sigma_t$ and
$\Sigma_t^-$ have the same $J_t$-value,
$|\Sigma_t^-|\leq|\Sigma_t|$.  Corollary~33 of \cite{EK} therefore gives
\begin{align*}
 \frac1{\cosh s}\sqrt{\frac{|\Sigma_s|}{\pi}}
 &\leq\frac1{\cosh t}\sqrt{\frac{|\Sigma_t^-|}{\pi}}
 \leq\frac1{\cosh t}\sqrt{\frac{|\Sigma_t|}{\pi}},
\end{align*}
which proves \eqref{eq:ordered-capillary-masses}.
\end{proof}

\begin{lemma}
\label{lem:positive-parameter-rigidity}
Fix $t\in(0,t_0)$ and form the one-ended completion
$(\widetilde M_t,\widetilde S_t)$ used in
Theorem~\ref{thm:capillary-mass}.  If
\begin{equation}
 m(\widetilde S_t)=\frac1{\cosh t}
 \sqrt{\frac{|\Sigma_t|}{\pi}},
 \label{eq:positive-parameter-equality}
\end{equation}
then $\Sigma_t$ is a round flat disk and $S_t^+$ is a portion of a catenoid
of waist radius $m(\widetilde S_t)$.
\end{lemma}

\begin{proof}
Let $\Sigma_r^-$ be a $J_r$-minimizer with least lateral support
area in the completed pair. Strict outer $J_t$-minimality of $\Sigma_t$,
Corollary~33 of \cite{EK}, and Proposition~37 of \cite{EK} give
\[
 m(\widetilde S_t)
 =\frac1{\cosh t}\sqrt{\frac{|\Sigma_t|}{\pi}}
 \leq\frac1{\cosh r}\sqrt{\frac{|\Sigma_r^-|}{\pi}}
 \leq\lim_{\rho\to\infty}
 \frac1{\cosh\rho}\sqrt{\frac{|\Sigma_\rho^-|}{\pi}}
 =m(\widetilde S_t)
\]
for every $r>t$. The equality cases in
\cite[Corollary~33 and Proposition~31]{EK} show that all $J_r$-minimizers
have the same lateral support area. For every $r\geq t$, they also allow us
to choose a $J_r$-minimizing union of round flat disks
$\widehat\Sigma_r$ such that $H_{\widetilde S_t}=0$ on
$\partial\widehat\Sigma_r$. Strict outer $J_t$-minimality gives
$\widehat\Sigma_t=\Sigma_t$.

We next prove that each $\widehat\Sigma_r$ is connected. If
$t\leq r_1<r_2$ and $\widehat\Sigma_{r_1}$ met
$\widehat\Sigma_{r_2}$, their boundary circles would meet transversely.
At an intersection point, the tangent directions of the two circles are
distinct principal directions of $\widetilde S_t$ with positive principal
curvatures, exactly as in the equality proof of \cite[Theorem~3]{EK}. This
contradicts $H_{\widetilde S_t}=0$ on the first contact circle. Hence the
surfaces $\widehat\Sigma_r$ are pairwise disjoint. Lemma~21 of \cite{EK},
together with equality of the lateral support areas, gives smooth dependence
on $r$ and constancy of the number of components. Lemmas~26 and~35 show that
the corresponding nested regions exhaust the connected distinguished
exterior. The equality proof of \cite[Theorem~3]{EK} therefore gives one
component at every parameter and
\begin{equation}
	S_t^+=\bigcup_{r\in[t,\infty)}\partial\widehat\Sigma_r.
	\label{eq:positive-level-sweep}
\end{equation}

Since $\widehat\Sigma_t=\Sigma_t$ and $\widehat\Sigma_t$ is connected,
$\Sigma_t$ is a round flat disk. Since
$H_{\widetilde S_t}=0$ on $\partial\widehat\Sigma_r$ for every $r\geq t$,
\begin{equation}
	H_{\widetilde S_t}=0\quad\text{on }S_t^+.
	\label{eq:positive-level-minimal-support}
\end{equation}
Let $n_t$ be the unit normal to $\Sigma_t$, and let $\eta_t$ be the radial
unit vector along $\partial\Sigma_t$ pointing toward the distinguished
exterior.  The capillary condition gives
\begin{equation}
	\nu_{\widetilde S_t}
	=-\tanh t\,n_t+\frac1{\cosh t}\eta_t
	\quad\text{on }\partial\Sigma_t.
	\label{eq:catenoid-cauchy-data}
\end{equation}
Choose the catenoid whose axis passes through the center of
$\partial\Sigma_t$ in the direction $n_t$ and whose waist radius is
$\cosh(t)^{-1}\sqrt{|\Sigma_t|/\pi}$. Choose its cross-section at axial
parameter $t$ to be $\partial\Sigma_t$. Its normal there is given by
\eqref{eq:catenoid-cauchy-data}, so it has the same Cauchy data as $S_t^+$.
The surface $S_t^+$ is minimal by
\eqref{eq:positive-level-minimal-support}. Boundary and interior unique
continuation for their local minimal graph representations therefore show
that $S_t^+$ is a portion of this catenoid. Finally,
\eqref{eq:positive-parameter-equality} gives
\[
\frac1{\cosh t}\sqrt{\frac{|\Sigma_t|}{\pi}}=m(\widetilde S_t),
\]
which proves the assertion.
\end{proof}

\begin{proposition}
	\label{prop:existing-outermost}
	Suppose that $D\in\mathcal A_S$ is outermost toward $\mathcal E_+$. Then
	\begin{equation}
		m_+(S)\geq\sqrt{\frac{|D|}{\pi}}
		\geq\sqrt{\frac{A_S}{\pi}}.
		\label{eq:outermost-comparison}
	\end{equation}
\end{proposition}

\begin{proof}
Outermostness implies that every component of $D$ is strongly stable.
Indeed, a negative first eigenvalue would yield a strictly mean-convex
deformation into $M_+(D)$. Barrier minimization as in
\cite[Lemmas~2.1 and~2.3]{Koerber} would then produce a free-boundary
minimal surface strictly outside $D$, contradicting outermostness.
Lemma~\ref{lem:stability-euler-characteristic}, with $t=0$, therefore shows
that every component of $D$ is a disk.

Consequently, $S_+(D)$ is a plane with finitely many open disks removed.
Cap each boundary component of $S_+(D)$ by a slight push-off of the
corresponding component of $D$, joined through a collar of $\partial D$.
Round the resulting corners away from $\partial D$. This produces a smooth
properly embedded asymptotically flat plane $\widehat S$. Let
$\widehat M$ be the component of $\mathbb R^3\setminus\widehat S$ containing
$M_+(D)$. Then the exterior of $D$ in $\widehat M$ is exactly $M_+(D)$,
with support $S_+(D)$. Hence $D$ remains outermost and free-boundary minimal in
$(\widehat M,\widehat S)$, and
\[
H_{\widehat S}=H_S\geq0\quad\text{on }S_+(D),
\qquad
H_{\widehat S}\in L^1(\widehat S),
\qquad
m(\widehat S)=m_+(S).
\]
Theorem~3 of \cite{EK} therefore gives
\[
 m_+(S)=m(\widehat S)\geq\sqrt{\frac{|D|}{\pi}}.
\]
The inequality $|D|\geq A_S$ follows from $D\in\mathcal A_S$.
\end{proof}

\begin{theorem}
	\label{thm:conditional-rigidity}
	Assume $A_S>0$ and
	\[
	m_+(S)=\sqrt{\frac{A_S}{\pi}}.
	\]
	Then there exists an outermost free-boundary minimal separator
	$D\in\mathcal A_S$ with $|D|=A_S$.  After a rigid motion,
	\[
	D=\{(y,0):|y|\leq m_+(S)\},
	\qquad
	S_+(D)=\mathrm{Cat}_{m_+(S)}^+.
	\]
\end{theorem}

\begin{proof}
The assumed equality gives $A_S=\pi m_+(S)^2$. By
Lemma~\ref{lem:ordered-capillary-separators}, the function
\[
t\longmapsto\frac1{\cosh t}\sqrt{\frac{|\Sigma_t|}{\pi}}
\]
is nondecreasing. Proposition~\ref{prop:capillary-separator} and
Theorem~\ref{thm:capillary-mass} give
\[
\frac{m_+(S)}{\cosh t}
\leq\frac1{\cosh t}\sqrt{\frac{|\Sigma_t|}{\pi}}
\leq m_+(S).
\]
Letting $t\downarrow0$ yields
\begin{equation}
\frac1{\cosh t}\sqrt{\frac{|\Sigma_t|}{\pi}}=m_+(S)
\quad \forall t\in(0,t_0).
\label{eq:constant-capillary-mass}
\end{equation}

For every $t\in(0,t_0)$, the associated one-ended completion satisfies
$m(\widetilde S_t)=m_+(S)$. Equation~\eqref{eq:constant-capillary-mass} and
Lemma~\ref{lem:positive-parameter-rigidity} therefore show that $\Sigma_t$
is a round flat disk and $S_t^+$ lies on a catenoid of waist radius
$m_+(S)$.
Moreover, \eqref{eq:ordered-capillary-regions} gives
\[
S_t^+\subset S_s^+\quad(0<s<t<t_0).
\]
The catenoidal portions overlap. By unique continuation, they belong to one
catenoid. After a rigid motion, this catenoid is
$\mathrm{Cat}_{m_+(S)}$, and $\partial\Sigma_t$ has radius
$m_+(S)\cosh t$. Moreover,
\[
\Sigma_t\longrightarrow
D:=\{(y,0):|y|\leq m_+(S)\}
\quad\text{smoothly as }t\downarrow0.
\]
Hence $D\subset\overline M$, and properness of $S$ gives
$\partial D\subset S$. Suppose $\operatorname{int}D\cap S\neq\varnothing$.
The one-sided strong maximum principle gives local coincidence of $D$ and
$S$, so $\operatorname{int}D\cap S$ is open in $S$. It is also closed.
Indeed, let $p\in S$ be a limit point. Then $p\in D$. If $p\in\partial D$,
local coincidence at the approaching points and continuity of the tangent
planes give $T_pS=T_pD$, contradicting the limiting orthogonal contact
angle. Thus $p\in\operatorname{int}D\cap S$.
Since $S$ is connected and noncompact, whereas $D$ is compact,
\[
\operatorname{int}D\cap S=\varnothing.
\]
Therefore
\[
D\cap S=\partial D,
\qquad
\langle\nu_D,\nu_S\rangle=0\quad\text{on }\partial D,
\]
so $D$ is free-boundary.

Let $M_t^+$ and $M_+(D)$ be the components containing $\mathcal E_+$ of
$M\setminus\Sigma_t$ and $M\setminus D$, respectively. We claim that
\begin{equation}
	M_+(D)=\bigcup_{0<t<t_0}M_t^+.
	\label{eq:limiting-outer-region}
\end{equation}
Indeed, for $x\in M_+(D)$, choose a compact path in $M_+(D)$ joining $x$
to a fixed point in the common $\mathcal E_+$-tail. The path has positive
distance from $D$, so $\Sigma_t\to D$ implies $x\in M_t^+$ for all
sufficiently small $t$. This proves
$M_+(D)\subset\bigcup_{0<t<t_0}M_t^+$. Conversely, let $x\in M_t^+$.
The disk $D$ lies on the side of $\Sigma_t$ away from $\mathcal E_+$, so
$x\notin D$. The nesting gives $x\in M_s^+$ for every $0<s<t$. Smooth
convergence with the distinguished side fixed implies that these components
converge locally in $M\setminus D$ to $M_+(D)$. Hence $x\in M_+(D)$, proving
the reverse inclusion.

For sufficiently small $t$, the smooth convergence $\Sigma_t\to D$ gives
an ambient isotopy supported in a fixed compact set that carries $D$ to
$\Sigma_t$.  Hence $M_t^+$ and $M_+(D)$ contain the same end tails.
Proposition~\ref{prop:capillary-separator} and \eqref{eq:topology} show that
$\mathcal E_+$ is the only end of $M_t^+$, and therefore of $M_+(D)$.
Thus $D$ separates $\mathcal E_+$ from all the remaining ends, and
\[
D\in\mathcal A_S,
\qquad
|D|=\pi m_+(S)^2=A_S.
\]
Moreover,
\[
S_+(D)=\overline{\bigcup_{0<t<t_0}S_t^+}
=\mathrm{Cat}_{m_+(S)}^+.
\]

Finally, let $Q$ be a connected component of a compact free-boundary
minimal surface in $\overline{M_+(D)}$. Since no closed
minimal surface exists in $\mathbb R^3$, $\partial Q\neq\varnothing$.
Let $x_3$ be the axial coordinate of $\mathrm{Cat}_{m_+(S)}$, normalized to
vanish on its waist. Since $Q$ is minimal, $\Delta_Qx_3=0$. Along
$\partial Q$, the free-boundary condition identifies the outward conormal
$\mu_Q$ with the outward normal of $\mathrm{Cat}_{m_+(S)}^+$, and hence
\[
\partial_{\mu_Q}x_3
=\langle e_3,\mu_Q\rangle
=-\tanh\left(\frac{x_3}{m_+(S)}\right).
\]
The divergence theorem gives
\begin{align*}
	0
	&=\int_Q\Delta_Qx_3
	=\int_{\partial Q}\partial_{\mu_Q}x_3
	=-\int_{\partial Q}
	\tanh\left(\frac{x_3}{m_+(S)}\right)ds.
\end{align*}
Since $x_3\geq0$ on $M_+(D)$, this identity gives $x_3=0$ on $\partial Q$.
The maximum principle then gives $x_3\equiv0$ on $Q$. Hence
\[
Q\subset\overline{M_+(D)}\cap\{x_3=0\}=D,
\qquad
 \partial Q
\subset S_+(D)\cap\{x_3=0\}=\partial D.
\]
As an embedded two-surface contained in the flat disk $D$, the interior of
$Q$ is relatively open in $\operatorname{int}D$. Compactness and
$\partial Q\subset\partial D$ make it relatively closed. Hence $Q=D$, and
$D$ is outermost toward $\mathcal E_+$.
\end{proof}

We conclude by verifying that the catenoid realizes equality.

\begin{proposition}
\label{prop:catenoid-model}
For $\mathrm{Cat}_m$, with either end designated as $\mathcal E_+$ and with
$M(\mathrm{Cat}_m)$ chosen as the component containing the axis, one has
\begin{equation}
 m_+(\mathrm{Cat}_m)=m,
 \qquad
 A_{\mathrm{Cat}_m}=\pi m^2.
 \label{eq:catenoid-model-data}
\end{equation}
Thus equality holds in \eqref{eq:main-inequality}.
\end{proposition}

\begin{proof}
On the upper graphical end,
\begin{align*}
 \psi_m(y)&=m\operatorname{arcosh}(|y|/m),
 &D\psi_m(y)&=\frac{m}{\sqrt{|y|^2-m^2}}\frac{y}{|y|},\\
 \frac1{2\pi R}\int_{|y|=R}y\cdot D\psi_m\,d\mathcal H^1
 &=\frac{mR}{\sqrt{R^2-m^2}}\longrightarrow m.
\end{align*}
Reflection across the waist plane gives the same mass for the other end.

The waist disk is admissible, so $A_{\mathrm{Cat}_m}\leq\pi m^2$. Conversely,
let $\Sigma\in\mathcal A_{\mathrm{Cat}_m}$ and let
$p:\mathbb R^3\to\mathbb R^2$ be horizontal projection. Since $\Sigma$
separates the two ends, its oriented boundary generates
$H_1(\mathrm{Cat}_m;\mathbb Z)$, and hence
\begin{equation}
 \deg(p|_\Sigma,z)=1\quad\text{for a.e. }z\in B_m,
 \qquad
 |\Sigma|\geq\int_\Sigma|\operatorname{Jac}_\Sigma p|\,dA
 \geq\int_{B_m}|\deg(p|_\Sigma,z)|\,dz=\pi m^2.
\end{equation}
Taking the infimum over $\Sigma$ proves \eqref{eq:catenoid-model-data}.
\end{proof}

\appendix

\section*{Appendix}
\setcounter{section}{1}
\setcounter{theorem}{0}
\setcounter{equation}{0}

The following lemma gives the required local smoothing.

\begin{lemma}
	\label{lem:localized-mean-convex-rounding}
	Let $N\subset\mathbb R^3$ be a domain. Suppose that for some open set $U$,
	\[
	\partial N\cap U=\Sigma_1\cup\Sigma_2,
	\]
	where $\Sigma_1$ and $\Sigma_2$ are smooth faces whose intersection
	$\Gamma=\Sigma_1\cap\Sigma_2$ is compact and smooth. Assume that the faces
	meet transversely along $\Gamma$ and that $N$ lies locally on their inner
	sides. With respect to the outward normals of $N$, suppose
	\[
	H_{\Sigma_1}\geq0,
	\qquad
	H_{\Sigma_2}\geq0
	\]
	near $\Gamma$. Then, for every neighborhood $V\Subset U$ of $\Gamma$,
	there exists a domain $\widetilde N$ differing from $N$ only inside $V$.
	Its boundary is smooth near $\Gamma$ and satisfies
	\[
	\widetilde N\triangle N\Subset V,
	\qquad
	\partial\widetilde N=\partial N
	\quad\text{outside }V,
	\qquad
	H_{\partial\widetilde N}\geq0
	\quad\text{on the modified part}.
	\]
\end{lemma}

\begin{proof}
Let $d_i$ be the signed distance to $\Sigma_i$, negative on $N$. After
shrinking $U$,
\[
N\cap U=\{d_1\leq0,\ d_2\leq0\},
\qquad
|\nabla d_1\wedge\nabla d_2|\geq c>0.
\]
Choose a smooth even convex function $\rho$ satisfying
\[
\rho(r)=|r|\quad\text{for }|r|\geq1,\qquad |\rho'|\leq1,\qquad \rho''\geq1-(\rho')^2,
\]
and define
\[
F_\varepsilon=\frac{d_1+d_2+\varepsilon\rho((d_1-d_2)/\varepsilon)}2.
\]
Thus $F_\varepsilon=\max\{d_1,d_2\}$ whenever
$|d_1-d_2|\geq\varepsilon$. Differentiation gives
\begin{align*}
\nabla F_\varepsilon
&=\frac{1+\rho'}2\nabla d_1+\frac{1-\rho'}2\nabla d_2,\\
D^2F_\varepsilon
&=\frac{1+\rho'}2D^2d_1+\frac{1-\rho'}2D^2d_2
+\frac{\rho''}{2\varepsilon}(\nabla d_1-\nabla d_2)\otimes(\nabla d_1-\nabla d_2),
\end{align*}
where $\rho'$ and $\rho''$ are evaluated at
$(d_1-d_2)/\varepsilon$. Uniform transversality gives
$|\nabla F_\varepsilon|\geq c>0$. Hence
$\Sigma_\varepsilon=\{F_\varepsilon=0\}$ is smooth, with outward mean
curvature
\begin{align*}
H_{\Sigma_\varepsilon}
={}&\frac{\operatorname{tr}_{T\Sigma_\varepsilon}\left(\frac{1+\rho'}2D^2d_1+\frac{1-\rho'}2D^2d_2\right)}{|\nabla F_\varepsilon|}
+\frac{\rho''|(\nabla d_1-\nabla d_2)^{T\Sigma_\varepsilon}|^2}{2\varepsilon|\nabla F_\varepsilon|}.
\end{align*}

On $\Sigma_\varepsilon\cap\{|d_1-d_2|\leq\varepsilon\}$, we have
\[
\frac{\rho''|(\nabla d_1-\nabla d_2)^{T\Sigma_\varepsilon}|^2}{2\varepsilon|\nabla F_\varepsilon|}
\geq\frac{c}{\varepsilon}\bigl(1-(\rho')^2\bigr)
\geq\frac{c}{\varepsilon}(1-|\rho'|).
\]
Moreover,
\[
|d_1|\leq\varepsilon(1-\rho'),\qquad |d_2|\leq\varepsilon(1+\rho').
\]
At every point
$x\in\Sigma_\varepsilon\cap\{|d_1-d_2|\leq\varepsilon\}$ where
$\rho'\leq0$, we have
\begin{align*}
	\operatorname{dist}(x,\Sigma_2)
	&=|d_2(x)|
	\leq\varepsilon(1+\rho'),\\
	\left|\frac{\nabla F_\varepsilon}{|\nabla F_\varepsilon|}-\nabla d_2\right|
	&\leq C|\nabla F_\varepsilon-\nabla d_2|
	\leq C(1+\rho').
\end{align*}
The preceding estimates, the boundedness of $D^2d_i$, and the mean convexity
of the two faces give
\begin{align*}
\frac{\operatorname{tr}_{T\Sigma_\varepsilon}\left(\frac{1+\rho'}2D^2d_1+\frac{1-\rho'}2D^2d_2\right)}{|\nabla F_\varepsilon|}
&\geq-C(1-|\rho'|),\\
H_{\Sigma_\varepsilon}
&\geq(1-|\rho'|)\left(\frac{c}{\varepsilon}-C\right)\geq0
\end{align*}
on $\Sigma_\varepsilon\cap\{|d_1-d_2|\leq\varepsilon\}$ for sufficiently
small $\varepsilon$. On the remaining part,
\[
F_\varepsilon=\max\{d_1,d_2\},
\qquad
\Sigma_\varepsilon\cap\{|d_1-d_2|\geq\varepsilon\}\subset\Sigma_1\cup\Sigma_2,
\]
so $H_{\Sigma_\varepsilon}\geq0$ there as well.

Finally, the preceding bounds and uniform transversality give
\[
\Sigma_\varepsilon\mathbin{\triangle}(\partial N\cap U)\subset\{|d_1|\leq2\varepsilon,\ |d_2|\leq2\varepsilon\}\subset\{x:\operatorname{dist}(x,\Gamma)\leq C\varepsilon\}.
\]
Choose $\varepsilon$ so that
$\{x:\operatorname{dist}(x,\Gamma)\leq C\varepsilon\}\subset V$. Replace
$N\cap U$ by $\{F_\varepsilon\leq0\}$. This defines the required domain
$\widetilde N$, which agrees with $N$ outside $V$ and has smooth mean-convex
boundary on the modified part.
\end{proof}

For completeness, we record the far-out Plateau construction used in
Lemma~\ref{lem:finite-slab-calibration}. The point specific to the
multiple-ended setting is that the minimizing disk remains in $\overline M$.

\begin{lemma}
\label{lem:asymptotic-plateau-disk}
For every sufficiently large regular value $\lambda$, the curve
\[
\gamma_\lambda=S_+\cap\partial B_\lambda
\]
bounds a unique absolutely area-minimizing disk
$D_\lambda\subset\overline M$. The disk is graphical, meets $S$
transversely, and is admissible.
\end{lemma}

\begin{proof}
For sufficiently large $\lambda$, the projection of $\gamma_\lambda$ bounds
a strictly convex planar domain. The asymptotic Plateau construction of
\cite[Section~2, preceding Lemma~12]{EK} then gives a unique absolutely
area-minimizing graph in $\mathbb R^3$ with boundary $\gamma_\lambda$. It
remains to place this graph in $\overline M$.

Choose an exhaustion of $M$ by intersections with large balls. Smooth the
resulting corners away from $\gamma_\lambda$ using
Lemma~\ref{lem:localized-mean-convex-rounding}. The truncated domains are
mean-convex, and \eqref{eq:topology} makes $\gamma_\lambda$ null-homotopic
in every sufficiently large truncation. Mean-convex Plateau theory therefore
gives an embedded least-area disk spanning $\gamma_\lambda$ in each such
domain. A fixed spanning disk gives a uniform area bound, while the
convex-hull property confines all these disks to a fixed compact set.
Standard compactness yields an embedded least-area disk in $\overline M$
with boundary $\gamma_\lambda$, and the maximum principle excludes interior
contact with $S$.

Rad\'o's theorem \cite{Rado} makes the limiting disk a graph over the planar
domain bounded by the projection of $\gamma_\lambda$. Uniqueness for the
minimal-graph equation identifies it with the graph constructed above and
also proves uniqueness among absolutely area-minimizing disks. The
asymptotic graphical estimates give transversality for sufficiently large
$\lambda$. Finally, \eqref{eq:topology} shows that the disk separates
$\mathcal E_+$ from all the remaining ends. Thus it is admissible.
\end{proof}

The following identities record the variational conventions for
$\mathscr F_t$. Passing to the complement reverses the normal and wall
orientations but describes the same capillary contact.

Let $G\subset\overline M$ and set
\[
\Sigma=\partial G\cap\mathring M,\qquad
E=\Omega(\Gamma_1)\setminus G,\qquad
W_G=\{x\in S:\operatorname{Tr}_S\chi_G(x)=1\}.
\]
Assume that $\Sigma$ is smooth and compact. Let $\nu_G$ be the outward unit
normal of $G$ along $\Sigma$. Along $\partial\Sigma$, let $\eta_\Sigma$ be the
unit conormal in $\Sigma$ directed away from the interior of $\Sigma$, and
let $\bar\eta_G$ be the unit conormal
in $S$ pointing from $W_G$ into $S\setminus W_G$. The orientations associated
with $E$ are
\[
\nu_E=-\nu_G,\qquad \bar\eta_E=-\bar\eta_G.
\]
Let $\{\Phi_s\}_{|s|<\varepsilon}$ be a smooth ambient variation preserving
$S$, and set
\[
\Phi_0=\operatorname{Id},\qquad \Phi_s(S)=S,\qquad
Y=\left.\partial_s\Phi_s\right|_{s=0}.
\]
For $t\in\mathbb R$, define
\[
 \mathscr F_t(G)=P(G;\mathring M)+\tanh t\,|W_G|.
\]

\begin{lemma}
\label{lem:complement-first-variation}
The first variation is
\begin{equation}
 \frac{d}{ds}\bigg|_{s=0}\mathscr F_t(\Phi_s(G))
 =\int_\Sigma H_\Sigma\langle Y,\nu_E\rangle\,dA
 +\int_{\partial\Sigma}
 \langle Y,\eta_\Sigma-\tanh t\,\bar\eta_E\rangle\,ds.
 \label{eq:appendix-first-variation}
\end{equation}
Consequently, if $G$ is $\mathscr F_t$-critical, then
\begin{equation}
 H_\Sigma=0,\qquad
 \langle\nu_E,\nu_S\rangle=-\tanh t.
 \label{eq:appendix-capillary-angle}
\end{equation}
\end{lemma}

\begin{proof}
The first-variation and transport formulas give
\begin{align*}
 \frac{d}{ds}\bigg|_{s=0}P(\Phi_s(G);\mathring M)
 &=\int_\Sigma H_\Sigma\langle Y,\nu_E\rangle\,dA
 +\int_{\partial\Sigma}\langle Y,\eta_\Sigma\rangle\,ds,\\
 \frac{d}{ds}\bigg|_{s=0}|W_{\Phi_s(G)}|
 &=-\int_{\partial\Sigma}\langle Y,\bar\eta_E\rangle\,ds.
\end{align*}
This proves \eqref{eq:appendix-first-variation}. Interior variations give
$H_\Sigma=0$. Along $\partial\Sigma$,
\[
 \eta_\Sigma^{T_S}
 =-\langle\nu_E,\nu_S\rangle\bar\eta_E.
\]
The boundary term vanishes for every $Y$ tangent to $S$ precisely when
$\langle\nu_E,\nu_S\rangle=-\tanh t$.
\end{proof}

\begin{lemma}
\label{lem:stability-euler-characteristic}
Let $\Sigma$ be a connected compact strongly stable minimal capillary surface with
\[
 \langle\nu_\Sigma,\nu_S\rangle=-\tanh t,
 \qquad t\geq0,
\]
and suppose $H_S\geq0$ along $\partial\Sigma$.  Then
\begin{equation}
 \frac12\int_\Sigma|h_\Sigma|^2\,dA
 +\cosh t\int_{\partial\Sigma}H_S\,ds
 \leq2\pi\chi(\Sigma).
 \label{eq:stability-topology-inequality}
\end{equation}
In particular, $\Sigma$ is a disk.
\end{lemma}

\begin{proof}
Testing the strong capillary stability inequality with constant normal speed
and using the boundary decomposition formula, the minimal-surface identity
\[
 2K_\Sigma=-|h_\Sigma|^2,
\]
and Gauss--Bonnet gives \eqref{eq:stability-topology-inequality}; see
\cite[Lemma~49]{EK}. Both terms on the left are nonnegative, so
$\chi(\Sigma)\geq0$. Since $\Sigma$ is a connected orientable surface with
boundary, it is a disk or an annulus.

Suppose that $\Sigma$ is an annulus. Since $\chi(\Sigma)=0$, equality holds in
\eqref{eq:stability-topology-inequality}, so $h_\Sigma=0$ and $H_S=0$ along
$\partial\Sigma$. The stability form vanishes on the constant
function, and its boundary Jacobi equation gives
$k_{\partial\Sigma}=0$. On the other hand, $h_\Sigma=0$ makes $\Sigma$ a
planar annulus, whose two embedded boundary curves satisfy
\[
 \int_{\partial\Sigma}|k_{\partial\Sigma}|\,ds\geq4\pi.
\]
This contradiction proves that $\Sigma$ is a disk.
\end{proof}

\begin{remark}
Lemma~\ref{lem:stability-euler-characteristic} is applied component by
component. It does not require the full capillary separator to be connected.
Thus the variational construction treats each component without assuming
that the full capillary separator is connected.
\end{remark}

\end{document}